\documentclass[12pt, reqno]{amsart}
\usepackage{amsmath, amsthm, amscd, amsfonts, amssymb, graphicx, color,enumerate}
\usepackage[bookmarksnumbered, colorlinks, plainpages]{hyperref}
\hypersetup{colorlinks=true,linkcolor=red, anchorcolor=green, citecolor=cyan, urlcolor=red, filecolor=magenta, pdftoolbar=true}

\newtheorem{theorem}{Theorem}[section]
\newtheorem{lemma}[theorem]{Lemma}
\newtheorem{proposition}[theorem]{Proposition}
\newtheorem{corollary}[theorem]{Corollary}
\theoremstyle{definition}

\theoremstyle{remark}

\numberwithin{equation}{section}

\newcommand{\real}{{\mathbb R}}

\newcommand{\A}{{\mathcal A}}
\newcommand{\B}{{\mathcal B}}
\newcommand{\E}{{\mathcal E}}
\newcommand{\F}{{\mathcal F}}
\newcommand{\G}{{\mathcal G}}
\newcommand{\M}{{\mathcal M}}
\newcommand{\N}{{\mathcal N}}
\newcommand{\8}{\infty}
\newcommand{\D}{{\mathcal D}}
\newcommand{\e}{\varepsilon}
\renewcommand{\t}{\tau}

\newcommand{\be}{\begin{eqnarray*}}
\newcommand{\ee}{\end{eqnarray*}}
\newcommand{\beq}{\begin{equation}}
\newcommand{\eeq}{\end{equation}}
\newcommand{\beqn}{\begin{equation*}}
\newcommand{\eeqn}{\end{equation*}}
\newcommand{\bsp}{\begin{split}}
\newcommand{\esp}{\end{split}}

\begin{document}

\setcounter{page}{1}

\title[Complex interpolation  noncommutative Hardy spaces]{ Complex interpolation  of noncommutative Hardy spaces associated semifinite von Neumann algebra  }

\author[Turdebek N. Bekjan]{Turdebek N. Bekjan }

\address{College of Mathematics and Systems Science, Xinjiang
University, Urumqi 830046, China.}
\email{\textcolor[rgb]{0.00,0.00,0.84}{bekjant@yahoo.com}}

\author{Kordan N. Ospanov}

\address{Faculty of Mechanics and Mathematics, L.N. Gumilyov Eurasian National University,  Astana 010008, Kazakhstan.}
\email{\textcolor[rgb]{0.00,0.00,0.84}{ospanov$_-$kn@enu.kz}}




\begin{abstract}
We  proved a complex interpolation theorem of noncommutative Hardy spaces  associated with semi-finite von Neumann algebras and extend
the Riesz type factorization to the semi-finite case.

\end{abstract}
\maketitle


\section{Introduction}

This paper deals with the complex interpolation and  Riesz type factorization of  noncommutative Hardy spaces associated with semifinite von Neumann algebras.
Let $\M$ be a semifinite (respectively, finite) von Neumann algebra equipped with  a normal faithful semifinite (respectively, finite)
trace $\tau$. Given $0<p\le\8$ we denote by $L^p(\M)$ the usual noncommutative
$L^p$-space associated with $(\M,\tau)$. Recall that $L^\8(\M)=\M$,
equipped with the operator norm.  The norm of $L^p(\M)$ will be
denoted by $\|\cdot\|_p$.
In this paper, $[K]_{p}$ denotes the closed linear
span of  $K$ in $L_{p}(\mathcal{M})$ (relative to the w*-topology for $p =\infty$) and $J(K)$ is the family of the adjoints of the elements of $ K$.

Let $\D$ be a von Neumann
 subalgebra of $\M$ such that the restriction of $\tau$ to $\D$ is still semifinite. Let $\E$ be the (unique) normal positive faithful
 conditional expectation of $\M$ with respect to $\D$ such that $\tau\circ\E=\tau$.  A {\it subdiagonal algebra} of $\M$ with respect to $\E$ (or $\D$) is a
w*-closed subalgebra $\A$ of $\M$ satisfying the following
conditions
 \begin{enumerate}[\rm (i)]

\item $\mathcal{A}+ J(\mathcal{A})$ is w*-dense in  $\mathcal{M}$,

\item $\mathcal{E}(xy)=\mathcal{E}(x)\mathcal{E}(y),\; \forall\;x,y\in
\mathcal{A},$

\item $\mathcal{A}\cap J(\mathcal{A})=\mathcal{D}.$
\end{enumerate}
The algebra $\D$ is called the
{\it diagonal} of $\A$.

 It is proved by Ji \cite{J} (respectively, Exel \cite{E}) that a  subdiagonal algebra $\A$ of a semifinite (respectively, finite) von Neumann algebra $\M$  is automatically  maximal, i.e.,  $\mathcal{A}$ is not properly contained
in any other subalgebra of $\mathcal{M}$ which is a subdiagonal algebra with respect to
$\mathcal{E}$. This maximality yields the
following useful characterization of $\A$:
 \beq\label{maxi-critere}
 \A=\{x\in\M\;:\; \t(xa)=0,\; \forall\; a\in\A_0\},
 \eeq
where $\A_0=\A\cap\ker\E$ (see \cite{A}).

 For $p<\8$ we define $H^p(\A)$ to be the
closure of $\A\cap L^{p}(\M)$ in $L^p(\M)$, and for $p=\8$ we simply set
$H^\8(\A)=\A$ for convenience. These are the so-called Hardy
spaces associated with $\A$. They are noncommutative extensions of
the classical Hardy spaces on the torus $\mathbb T$. For the finite case,  most results on
the classical Hardy spaces on the torus have been established in
this noncommutative setting. We refer to \cite{A,BL1,BL2,BL3,BL4,MW,PX}. Here we mention only the interpolation
theorem and the Riesz
factorization theorem directly related with the objective of this paper.

In \cite{P}, Pisier gave a new proof of the interpolation theorem
of Peter Jones (see \cite{Jo} or \cite{BS}, p.414). He obtained the complex case of Peter Jones' theorem as a
consequence of the real case. The Pisier's method does  extend to the noncommutative case and the case of Banach space valued $H^p$-spaces  (see  $\S2$ in \cite{P}). Pisier and Xu \cite{PX} obtained the following
noncommutative version of Peter  Jones' theorem.
\begin{theorem} \label{interpolation-finite-hp}{\rm(Pisier/Xu)} Let $\M$ be a  finite von Neumann algebra, and let $\A$ be a  subdiagonal algebra of $\M$. If  $0<p_0,\ p_1\le\8$, $0<\theta<1$ and $\frac{1}{p}=\frac{1-\theta}{p_0}+\frac{\theta}{p_1}$, then
$$
(H^{p_0}(\A), H^{p_1}(\A))_\theta=H^{p}(\A)
$$
and
$$
\|x\|_p\le \|x\|_{(H^{p_0}(\A), H^{p_1}(\A))_\theta}\le C\|x\|_p,\quad \forall x\in H^{p}(\A),
$$
where $C$  depends only on $p_0, p_1$ and $\theta$.
\end{theorem}
The first-named author \cite{B1}, using Pisier's method, proved that the real case of Peter Jones' theorem   for  noncommutative Hardy spaces associated with semifinite von Neumann algebras holds (also see \cite{STZ}). The main purpose of the present paper is to prove the complex case of Peter Jones' theorem for noncommutative Hardy spaces  associated with semifinite von Neumann algebras, i.e. to extend the above Pisier/Xu's theorem to the semifinite case.

The Riesz
factorization  asserts that $H^p(\A)=H^q(\A)\cdot H^r(\A)$
for any $0< p, q,r\le\8$ such that $1/p=1/q+1/r$. More
precisely, given $x\in H^p(\A)$ and $\e>0$ there exist $y\in
H^q(\A)$ and $z\in H^r(\A)$ such that
 $$x=yz\quad\mbox{and}\quad \|y\|_q\,\|z\|_r\le\|x\|_p +\e.$$
This result is proved in \cite{S} for $p=q=2$, in
\cite{MW} for $r=1$, and independently in \cite{L}
and in \cite{PX} for $1\le p, q,r\le\8$, and in \cite{BX} for the general case as above. We prove a similar result in the semifinite case.

The organization of the paper is as follows. In Section 2, we
prove the complex interpolation theorem. Section 3 is devoted to the Riesz type factorization theorem. In Section 4, we give some applications of the complex interpolation theorem.

We will keep all previous notations throughout the paper. Unless
explicitly stated otherwise,  $\M$ always denotes a semifinite von Neumann algebra equipped with a
semifinite  faithful trace $\tau$, and $\A$   a subdiagonal algebra
of  $\M$.

\section{Complex interpolation}

Let $S$ (respectively, $\overline{S}$ ) denote the open strip $\{z:\;0<\mathrm{Re}z<1\}$ (respectively, the closed strip $\{z:\;0\le \mathrm{Re}z\le1\}$ ) in the complex plane $\mathbb{C}$.
Let $A(S)$ be the space of complex valued functions, analytic in $S$ and continuous and bounded
in  $\overline{S}$.  Let $(X_0, X_1)$ be a compatible couple of complex quasi-Banach spaces.  Let us denote by $\F_0(X_0,
X_1)$ the family of functions of the form $f(z)=\sum_{k=1}^n f_k(z)x_k$ with $f_k$ in $A(S)$ and $x_k$ in $X_0 \cap X_1$.
We equip $\F_0(X_0, X_1)$ with the norm:
 $$
 \big\|f\big\|_{\F_0(X_0, X_1)}=\max\big\{
 \sup_{t\in\real}\big\|f(it)\big\|_{X_0}\,,\;
 \sup_{t\in\real}\big\|f(1+it)\big\|_{X_1}\big\}.
 $$
Then  $\F_0(X_0, X_1)$ becomes
a quasi-Banach  space. Let $0<\theta<1$. The complex interpolation norm on
$X_0\cap X_1$ is defined for $\theta$ by
 $$\|x\|_{(X_0,\; X_1)_\theta}=\inf\big\{\big\|f\big\|_{\F_0(X_0, X_1)}\;:\; f(\theta)=x,\;
 f\in\F_0(X_0, X_1)\big\}.$$
We denote by $(X_0,\; X_1)_\theta$ the completion of $(X_0\cap X_1,\;\|\cdot\|_{(X_0,\; X_1)_\theta})$.

It is well-known that if $(X_0, X_1)$ is a compatible couple of complex Banach spaces, then the definition of the interpolation
spaces here coincides with that of \cite{C}.

Let $e$ be  a projection  in $\mathcal{D}$. We set
\be
\mathcal{M}_{e}=e\mathcal{M}e,\quad \mathcal{A}_{e}=e\mathcal{A}e,\quad
\mathcal{D}_{e}=e\mathcal{D}e,
\ee
and let $\mathcal{E}_{e}$ be the restriction of $\E$  to $\mathcal{M}_{e}$.
Using Lemma 3.1 in \cite{B1} we obtain that
  $\mathcal{A}_{e}$ is a subdiagonal algebra of
 $\mathcal{M}_{e}$ with respect to $\mathcal{E}_{e}$ (or $\mathcal{D}_{e}$).

Since  $\mathcal{D}$ is semifinite, we can choose an increasing family of
$\{e_{i}\}_{i\in I}$ of $\tau$-finite ($\tau(e_{i})<\8$) projections in $\mathcal{D}$ such that
$e_{i}\rightarrow1$  strongly, where 1 is identity of $\mathcal{M}$ (see Theorem 2.5.6 in \cite{Sa}). Throughout,  $\{e_{i}\}_{i\in I}$ will be used to indicate this net.

\begin{lemma}\label{lem:saito-type} Let $0<p_0\le p\le p_1\le\8$. Then
\beq\label{eq:saito-type}
H^p(\A)=\left(H^{p_0}(\A)+H^{p_1}(\A)\right)\cap L^p(\M).
\eeq
\end{lemma}
\begin{proof} It is clear that $H^p(\A)\subseteq\left(H^{p_0}(\A)+H^{p_1}(\A)\right)\cap L^p(\M)$ (see Theorem 6.3 in \cite{B1}). Conversely, if $x\in \left(H^{p_0}(\A)+H^{p_1}(\A)\right)\cap L^p(\M)$, then
$$
e_ixe_i\in \left(e_iH^{p_0}(\A)e_i+e_iH^{p_1}(\A)e_i\right)\cap e_iL^p(\M)e_i,\quad \forall i\in I.
$$ By lemma 3.1 in \cite{B1}, we have that
$$
e_iH^{p_0}(\A)e_i=H^{p_0}(\A_{e_i}),\quad e_iH^{p_1}(\A)e_i=H^{p_1}(\A_{e_i}),\quad e_iL^p(\M)e_i=L^{p}(\M_{e_i}),\quad \forall i\in I.
$$
Hence,
$$
e_ixe_i\in \left(H^{p_0}(\A_{e_i})+H^{p_1}(\A_{e_i})\right)\cap L^{p}(\M_{e_i})\subset H^{p_0}(\A_{e_i})\cap L^{p}(\M_{e_i}),\quad \forall i\in I.
$$
Since $\A_{e_i}$ is a subdiagonal algebra of the finite von Neumann algbra $\M_{e_i}$, by Proposition 3.3 \cite{BX}, we get
$$
e_ixe_i\in H^{p}(\A_{e_i})= e_iH^{p}(\A)e_i\subset H^{p}(\A),\quad \forall i\in I.
$$
Note that since $e_{i}\rightarrow1$  strongly, we get
$\lim_{i}\|xe_{i}- x\|_{p}=0$ and $\lim_{i}\|e_{i}x-x\|_{p}=0$  (cf. Lemma 2.3
in \cite{Ju}). Therefore,
\beq\label{eq:convergence}
\lim_{i}\|x-e_{i}xe_{i}\|_{p}\leq\lim_{i}\|x-xe_{i}\|_{p}+\lim_{i}\|(x-e_{i}x)e_{i}\|_{p}=0.
\eeq
Thus $x\in H^{p}(\A)$. This gives the desired result.
\end{proof}

\begin{theorem} \label{thm:interpolation-semi-hp} Let $0<p_0,\ p_1\le\8$ and $0<\theta<1$. If  $\frac{1}{p}=\frac{1-\theta}{p_0}+\frac{\theta}{p_1}$, then
$$
(H^{p_0}(\A), H^{p_1}(\A))_\theta=H^{p}(\A)
$$
and
$$
\|x\|_p\le \|x\|_{(H^{p_0}(\A), H^{p_1}(\A))_\theta}\le C\|x\|_p,\quad \forall x\in H^{p}(\A),
$$
where $C$  depends only on $p_0, p_1$ and $\theta$.

\end{theorem}

\begin{proof}
 By Theorem 4.1 in \cite{Xu1},  we have
 $(L^{p_0}(\M),
L^{p_1}(\M))_{\theta}=L^{p}(\M)$
with equal norms. Hence,
$$
(H^{p_0}(\A), H^{p_1}(\A))_\theta\subset (L^{p_0}(\M),
L^{p_1}(\M))_{\theta}=L_{p}(\M).
$$
 On the other hand,
 $$
 (H^{p_0}(\A), H^{p_1}(\A))_\theta \subset H^{p_0}(\A)+H^{p_1}(\A).
 $$
Using \eqref{eq:saito-type} we obtain that  $(H^{p_0}(\A), H^{p_1}(\A))_\theta \subset H_{p}(\A)$ and
\beq\label{eq:semip1}
\|x\|_p\le\|x\|_{(H^{p_0}(\A), H^{p_1}(\A))_\theta },\quad \forall x\in (H^{p_0}(\A), H^{p_1}(\A))_\theta .
\eeq
Recall that $\A_{e_i}$ is  a  subdiagonal algebra of the finite von Neumann algbra $\M_{e_i}\;(i\in I)$.  By Theorem \ref{interpolation-finite-hp}, we have that
 $$
(H^{p_0}(\A_{e_i}),H^{p_1}(\A_{e_i}))_\theta=H^{p}(\A_{e_i}), \quad\forall i\in I.
 $$
Hence
 $$
 H_{p}(\A_{e_i})=(H^{p_0}(\A_{e_i}),H^{p_1}(\A_{e_i}))_\theta\subset (H^{p_0}(\A), H^{p_1}(\A))_\theta
 $$
 and
 \beq\label{eq:semip2}
  \|a\|_{(H^{p_0}(\A), H^{p_1}(\A))_\theta }\le C\|a\|_p,\quad \forall a\in H_{p}(\A_{e_i}),\quad \forall i\in I.
 \eeq
 Let $a\in H^{p}(\A)$. Similar to \eqref{eq:convergence}, we deduce  that
 $$
\lim_{i}\| e_iae_i -a\|_{p}=0.
 $$
 Using \eqref{eq:semip2} we obtain that
 $$
\|e_iae_i-e_jae_j\|_{(H^{p_0}(\A), H^{p_1}(\A))_\theta }\le C\|e_iae_i-e_jae_j\|_p,\quad \forall i,j\in I,
$$
since $\{e_{i}\}_{i\in I}$ is increasing and $I$ is  a directed set.
Thus $\{e_iae_i\}_{i\in I}$ is a Cauchy net in $(H^{p_0}(\A), H^{p_1}(\A))_\theta $.
Therefore, there exists  $y$ in $(H^{p_0}(\A), H^{p_1}(\A))_\theta $ such that
$$
\lim_{i}\|e_iae_i-y\|_{(H^{p_0}(\A), H^{p_1}(\A))_\theta }=0
$$
By \eqref{eq:semip1}, $\lim_{n\rightarrow\8}\|e_iae_i-y\|_{p}=0$. Therefore, $y=a$. Thus
$$
  \|a\|_{(H^{p_0}(\A), H^{p_1}(\A))_\theta }\le C\|a\|_p,\quad \forall a\in H^{p}(\A).
$$
This completes the proof.
\end{proof}

\section{Riesz  type factorization}

Let $0<r,\;p,\;q\le\8,\;\frac{1}{r}=\frac{1}{p}+\frac{1}{q}$.
$H^p(\A)\odot H^q(\A)$ is defined as the space of all operators $x$ in $H^r (\A)$ for which there exist a  sequence $(a_n)_{n \ge 1}$ in $H^p(\A)$ and a sequence $(b_n)_{n \ge 1}$ in $H^q(\A)$ such that $\sum_{n\ge1}\|a_n\|_p^{\min\{1,r\}}\|b_n\|_q^{\min\{1,r\}}<\8$ and
$x=\sum_{n\ge1}a_nb_n$.
For each $x\in H^p(\A)\odot H^q(\A)$, we define
$$
\|x\|_{p\odot q}^{\min\{1,r\}}=\inf\left\{\sum_{n\ge1}\|a_n\|_p^{\min\{1,r\}}\|b_n\|_q^{\min\{1,r\}}\right\}.
$$
where the infimum runs over all possible factorizations of $x$ as above.

It is clear that   $ H^p(\A)\odot H^q(\A)\subset H^r(\A)$ and
\beq\label{eq:multiplication-norm}
\|x\|_r\le \|x\|_{p\odot q},\quad \forall x\in H^p(\A)\odot H^q(\A).
\eeq

\begin{proposition} Let $0<r,\;p,\;q\le\8,\;\frac{1}{r}=\frac{1}{p}+\frac{1}{q}$. If $r\ge1$ (respectively, $r<1$), then $H^p(\A)\odot H^q(\A)$ is a Banach space (respectively, quasi-Banach space).
\end{proposition}
\begin{proof}
Suppose that  $r\ge1$. If $x,y\in H^p(\A)\odot H^q(\A)$, then for $\varepsilon>0$, there exist  sequences $(a_n)_{n \ge 1},\;(c_n)_{n \ge 1}$ in $H^p(\A)$ and  sequences $(b_n)_{n \ge 1},\;(d_n)_{n \ge 1}$ in $H^q(\A)$ such that
$$
\sum_{n\ge1}\|a_n\|_p\|b_n\|_q<\8,\quad\sum_{n\ge1}\|c_n\|_p\|d_n\|_q<\8,\quad x=\sum_{n\ge1}a_nb_n,\quad y=\sum_{n\ge1}c_nd_n
$$
and
$$
\sum_{n\ge1}\|a_n\|_p\|b_n\|_q<\|x\|_{p\odot q}+\frac{\varepsilon}{2},\quad \sum_{n\ge1}\|c_n\|_p\|d_n\|_q<\|y\|_{p\odot q}+\frac{\varepsilon}{2}.
$$
Hence, $x+y=\sum_{n\ge1}(a_nb_n+c_nd_n)$ and $\sum_{n\ge1}(\|a_n\|_p\|b_n\|_q+\|c_n\|_p\|d_n\|_q)<\8$. Therefore, $x+y\in H^p(\A)\odot H^q(\A)$ and
$$
\|x+y\|_{p\odot q}\le\sum_{n\ge1}\|a_n\|_p\|b_n\|_q+\sum_{n\ge1}\|c_n\|_p\|d_n\|_q<\|x\|_{p\odot q}+\|y\|_{p\odot q}+\varepsilon.
$$
Thus $\|x+y\|_{p\odot q}\le|x\|_{p\odot q}+\|y\|_{p\odot q}$.
It is trivial that
$$
\lambda x\in H^p(\A)\odot H^q(\A),\;\|\lambda x\|_{p\odot q}=|\lambda|\|x\|_{p\odot q},\quad \forall x\in H^p(\A)\odot H^q(\A),\;\forall\lambda\in\mathbb{C}.
$$
On the other hand, \eqref{eq:multiplication-norm} insures that $\|x\|_{p\odot q}>0$ for any non-zero $x\in H^p(\A)\odot H^q(\A)$. Thus $\|\;\|_{p\odot q}$ is a norm
on  $ H^p(\A)\odot H^q(\A)$.

Let us prove the completeness of the norm $\|\;\|_{p\odot q}$. Let $(x_k)_{k \ge 1}$ be a sequence in $ H^p(\A)\odot H^q(\A)$ such that $\sum_{k\ge1}\|x_k\|_{p\odot q}<\8$. By \eqref{eq:multiplication-norm}, $\sum_{k\ge1}\|x_k\|_{r}<\8$. It follows that $\sum_{k=1}x_k\in H^r (\A)$. Set $x=\sum_{k\ge1}x_k$. Let $\varepsilon>0$ be given and select a sequence $(a_n^{(k)})_{n \ge 1}$ in $H^p(\A)$ and  a sequence $(b_n^{(k)})_{n \ge 1}$ in $H^q(\A)$ such that
$x_k=\sum_{n\ge1}a_n^{(k)}b_n^{(k)}$ and
$$
\sum_{n\ge1}\|a_n^{(k)}\|_p\|b_n^{(k)}\|_q<\|x_k\|_{p\odot q}+\frac{\varepsilon}{2^k}
$$
 for every $k$. Since
 $$
 \sum_{k\ge1}\sum_{n\ge1}\|a_n^{(k)}\|_p\|b_n^{(k)}\|_q\le\sum_{k\ge1}\|x_k\|_{p\odot q}+\varepsilon<\8,
 $$
 we get
 $x=\sum_{k\ge1}\sum_{n\ge1}a_n^{(k)}b_n^{(k)}$, and so $x\in H^p(\A)\odot H^q(\A)$. On the other hand,
$$
\|x-\sum_{k=1}^Nx_k\|_{p\odot q}=\|\sum_{k\ge N+1}x_k\|_{p\odot q}\le\sum_{k\ge N+1}\|x_k\|_{p\odot q}\rightarrow 0\quad\mbox{as}\;N\rightarrow\8.
$$
Thus $x=\sum_{k\ge1}x_k$ in $H^p(\A)\odot H^q(\A)$.

The case $0<r<1$ follows analogously.

\end{proof}

\begin{theorem}\label{thm:hp-multiplication} Let $0<r,\;p,\;q\le\8,\;\frac{1}{r}=\frac{1}{p}+\frac{1}{q}$. Then
\beq
H^r(\A)=H^p(\A)\odot H^q(\A).
\eeq
\end{theorem}
\begin{proof} We prove only the case $r\ge1$. The proof of the case $r<1$ is
similar.
By Lemma 3.1 in \cite{B1}, we know that $e_{i}H^r(\A)e_{i}=H^r(\A_{e_i})\;( i\in I)$. Since $\mathcal{A}_{e_i}$ is a subdiagonal algebra of the finite von Neumann algebra
 $\mathcal{M}_{e_i}$, by Theorem 3.4 in \cite{BX}, for $x_i\in H^r(\A_{e_i})$ and $\e>0$ there exist $y_\e\in H^p(\A_{e_i})$ and
$z_\e\in H^q(\A_{e_i})$ such that
 $$x_i=y_\e z_\e\quad\mbox{and}\quad \|y_\e\|_p\,\|z_\e\|_q\le \|x_i\|_r+\e.$$
Hence, $x_i\in H^p(\A)\odot H^q(\A)$ and
$$
\|x_i\|_{p\odot q}\le \inf_{\e}\|y_\e\|_p \|z_\e\|_q=\|x_i\|_r,
$$
so $e_{i}H^r(\A)e_{i}\subset H^p(\A)\odot H^q(\A)\,( i\in I)$.
Using \eqref{eq:multiplication-norm} we obtain that
\beq\label{equal:multiplication-norm}
\|x_i\|_{p\odot q}= \|x_i\|_r,\quad \forall x_i\in e_{i}H^r(\A)e_{i},\quad \forall i\in I.
\eeq

Let $x\in H^{r}(\A)$.  Similar to \eqref{eq:convergence}, we have that
$$
\lim_{i}\|x-e_{i}xe_{i}\|_{r}=0.
$$
Notice that $\{e_{i}\}_{i\in I}$ is increasing, so $e_{i}xe_{i}-e_{j}xe_{j}\in e_{i}H^r(\A)e_{i}$ for any $i\ge j$. Applying \eqref{equal:multiplication-norm} we obtain that
$$
\|e_{i}xe_{i}-e_{j}xe_{j}\|_{p\odot q}=\|e_{i}xe_{i}-e_{j}xe_{j}\|_r,\quad i\ge j.
$$
Similar to the proof of Theorem \ref{thm:interpolation-semi-hp}, we  obtain  that
 $H^p(\A)\odot H^q(\A)= H^{r}(\A)$.
\end{proof}

By Theorem \ref{thm:interpolation-semi-hp} and \ref{thm:hp-multiplication}, we obtain the following result.

\begin{corollary} Let $0<p_0,\ p_1\le\8$ and $0<\theta<1$. Then
\be
(H^{p_0}(\A),H_{p_1}(\A)_\theta =H^{\frac{p_0}{1-\theta}}(\A)\odot H^{\frac{p_1}{\theta}}(\A)
\ee
with equivalent norms, where $\frac{1}{p}=\frac{1-\theta}{p_0}+\frac{\theta}{p_1}$.
\end{corollary}

\section{Application}

In Section 4 of \cite{MW2},  Marsalli has shown that if $x\in \mathrm{Re}(\A)$  then there exists a uniquely
determined $\tilde{x}\in \mathrm{Re}(\A)$  such that $x+i\tilde{x}\in \A$ and $\E(\tilde{x})=0$. The  map $\sim \mathrm{Re}(\A)\rightarrow \mathrm{Re}(\A)$ is a real linear map and it is called the conjugation map. If $1<p<\8$, one can extend the conjugation map $\sim \mathrm{Re}(\A)\rightarrow \mathrm{Re}(\A)$ to a bounded real linear
map $\sim: L^p(\M)^{sa}\rightarrow L^p(\M)^{sa}$. Its complexification $\overline{\sim}$  is
a bounded operator on $L^p(\M)$ (see Theorem 5.4 in \cite{MW} or Theorem 2 in \cite{R}). Note that if $x \in L^p(\M)$, then
$x+i\overline{\widetilde{x}}\in H^p(\A)$ and $\E(\overline{\widetilde{x}})=0$. In section 4 of \cite{B1}, this result is extended to the semifinite  case.

For $x,y\in\M$, we define
\be
x+\overline{\widetilde{y}}=\{(u,v):\;u,v\in\M, e_ixe_i+\overline{\widetilde{e_iye_i}}=e_iue_i+\overline{\widetilde{e_ive_i}},\;\forall i\in I\}.
\ee
 Let $BMO(\M)$ be the
completion of the set $ \{x+\overline{\widetilde{y}}:\; x,y \in \mathcal{M}\}$
 with norm
 \be
\|x+\overline{\widetilde{y}}\|_{BMO}=\inf\left\{\|u\|_{\infty}+\|v\|_{\infty}:\; (u,v) \in x+\overline{\widetilde{y}}\right\}.
\ee
Then
\beq\label{dual-H^1}
H^{1}(\mathcal{A})^*=BMO(\mathcal{M})
\eeq
under the pairing
  $$
\langle a,\;x+\overline{\widetilde{y}}\rangle=\lim_{i}\tau(a(e_{i}xe_{i}+\overline{\widetilde{e_{i}ye_{i}}})^{\ast}),\quad \forall  a\in H^{1}(\mathcal{A}),\;\forall x,\; y\in\M.
  $$
This identification satisfies
\beq\label{dual-H^1}
\|x+\overline{\widetilde{y}}\|_{BMO}\le \|x+\overline{\widetilde{y}}\|_{H^{1}(\mathcal{A})^*}\le2\|x+\overline{\widetilde{y}}\|_{BMO}.
\eeq
For more details, see \cite{B1}.

We use the duality theorem, the reiteration theorem (Theorem 4.5.1, Theorem 4.6.1 in \cite{BL}) and Wolff's theorem (Theorem 2 in \cite{W}) to obtain

\begin{proposition}\label{inter:H-BMO} Let $0<p_0\le p<\8$ and $\frac{1}{p}=\frac{1-\theta}{p_0}$. Then
\be
H^p(\A)=(H^{p_0}(\A),BMO(\M))_\theta
\ee
with equivalent norms.
\end{proposition}

\begin{lemma} Let $T:\;BMO(\M)\rightarrow BMO(\M)$ be defined by $T(x+\overline{\widetilde{y}})=\E(y)-y+\overline{\widetilde{x}}$. Then $T$ is a bounded linear map
with norm at most 2.
\end{lemma}
\begin{proof} Let $(u,v)\in x+\overline{\widetilde{y}}$. Then $e_ixe_i+\overline{\widetilde{e_iye_i}}=e_iue_i+\overline{\widetilde{e_ive_i}}$ for all $i\in I$.
Since $e_iye_i,e_ive_i\in L^1(M_{e_i})$, by Lemma 5.1 and Proposition 5.3 in \cite{MW},
\be
\overline{\widetilde{e_ixe_i}}+e_i(\E(y)-y)e_i=\overline{\widetilde{e_iue_i}}+e_i(\E(v)-v)e_i,\;\forall i\in I.
\ee
Thus $(\E(v)-v,u)\in \E(y)-y+\overline{\widetilde{x}}$. Hence,
\be
\|\E(y)-y+\overline{\widetilde{x}}\|_{BMO}\le\|\E(v)-v\|_\8+\|u\|_\8 \le2\|v\|_\8+\|u\|_\8,
\ee
and so $\|T(x+\overline{\widetilde{y}})\|_{BMO}\le2\|x+\overline{\widetilde{y}}\|_{BMO}$. Thus the proof is complete.
\end{proof}

Since $T$ is a continuous liner extension of $\overline{\sim}$, $T$ will be denoted still by $\overline{\sim}$.

Let $P : L^2(\M) \rightarrow H^2(\M)$ be the canonical projection. Then
\be
P(x)=\frac{x+\E(x)+i\overline{\widetilde{x}}}{2},\qquad \forall x\in L^2(\M).
\ee
\begin{lemma}\label{lemm:M-BMO}
Let $P': \M\rightarrow BMO(\M)$ be the map taking $x$ to $\frac{x+\E(x)+i\overline{\widetilde{x}}}{2}$. Then
$P'$ is a continuous linear extension of $P$, still denoted by $P$.
\end{lemma}
\begin{proof}
\be
\|P'(x)\|_{BMO}=\|\frac{x+\E(x)+i\overline{\widetilde{x}}}{2}\|_{BMO}\le\|\frac{x+\E(x)}{2}\|_\8+\|\frac{\overline{\widetilde{x}}}{2}\|_\8\le\frac{3}{2}\|x\|_\8.
\ee
Thus the proof is complete.
\end{proof}

We  proved the interpolation theorem (Proposition \ref{inter:H-BMO}) with the aid of the  boundedness of $P$ on $L^p(\M)$ for all $1<p<\8$. Now we use the  interpolation theorem to prove that $P$ is a bonded operator  on $L^p(\M)$ for all $1<p<\8$. Indeed, using the $L^2$-boundedness of $P$, Lemma \ref{lemm:M-BMO} and Proposition \ref{inter:H-BMO}, we easily check that the  boundedness of $P$ on $L^p(\M)$ for all $2<p<\8$ (and by duality for the case $1<p<2$).

For $0<p\le \8$, we define
\be
b_p(x+\overline{\widetilde{y}})=\inf_{(u,v)\in x+\overline{\widetilde{y}}}\sup_{h\in H^p(\A),\|h\|_p\le1}\|(u-\mathrm{i}(1-\E)(v))h^*\|_{p}.
\ee

\begin{theorem} Let $0<p\le\8$.
\begin{enumerate}[\rm(i)]
  \item If $1\le p\le\8$, then $\|x+\overline{\widetilde{y}}\|_{BMO}\le b_p(x+\overline{\widetilde{y}})\le 2\|x+\overline{\widetilde{y}}\|_{BMO}$
  \item If $0<p<1$, then for every $(u,v)\in x+\overline{\widetilde{y}}$ with $\|u\|_{\infty}+\|v\|_{\infty}\le 2\|x+\overline{\widetilde{y}}\|_{BMO}$,
  \be
  C\|x+\overline{\widetilde{y}}\|_{BMO}\le  \sup_{h\in H^p(\A),\|h\|_p\le1}\|(u-\mathrm{i}(1-\E)(v))h^*\|_{p}\le 2\|x+\overline{\widetilde{y}}\|_{BMO},
  \ee
  where $C$ is a positive constant.
 \end{enumerate}
\end{theorem}
\begin{proof}  It is clear that for all $0<p\le\8$,
\be
b_p(x+\overline{\widetilde{y}})\le2\|x+\overline{\widetilde{y}}\|_{BMO}.
\ee

(i) Let $x,y\in\M$. If $ (u,v)\in x+\overline{\widetilde{y}}$ and $a\in \mathcal{A}_{e_{i}}$, then
\be\begin{array}{rl}
\tau((e_{i}xe_{i}+\overline{\widetilde{e_{i}ye_{i}}})a^{\ast})&=\tau((e_{i}ue_{i}+e_{i}\overline{\widetilde{v}}e_{i})a^{\ast})\\
&=\tau((e_{i}ue_{i}+e_{i}\overline{\widetilde{v}}e_{i})e_{i}a^{\ast}e_{i})\\
&=\tau((e_{i}ue_{i}-\mathrm{i}(1-\E)(e_{i}ve_{i}))e_{i}a^{\ast}e_{i})\\
&=\tau((u-\mathrm{i}(1-\E)(v))e_{i}a^{\ast}e_{i})
 \end{array}
 \ee
(see Lemma 1 and 3 of \cite{MW1}, and Lemma 4.1 in \cite{B1}). Since $ H^{1}(\mathcal{A})$ is closure
 of $\cup_{i\in I}\mathcal{A}_{e_{i}}$ (see Lemma 3.2 in \cite{B1}), by \eqref{dual-H^1},
\be\begin{array}{rl}
      \|x+\overline{\widetilde{y}}\|_{BMO}&\le \sup_{h\in H^1(\A),\|h\|_1\le1}|\langle a,\;x+\overline{\widetilde{y}}\rangle|  \\
      &= \sup_{a\in \cup_{i\in I}\mathcal{A}_{e_{i}},\|a\|_1\le1}|\langle a,\;x+\overline{\widetilde{y}}\rangle|  \\
      & =\sup_{a\in \cup_{i\in I}\mathcal{A}_{e_{i}},\|a\|_1\le1}|\lim_{i}\tau(a(e_{i}xe_{i}+\overline{\widetilde{e_{i}ye_{i}}})^{\ast})|\\
        &=\sup_{a\in \cup_{i\in I}\mathcal{A}_{e_{i}},\|a\|_1\le1}|\lim_{i}\tau((e_{i}xe_{i}+\overline{\widetilde{e_{i}ye_{i}}})a^{\ast})|\\
         &=\sup_{a\in \cup_{i\in I}\mathcal{A}_{e_{i}},\|a\|_1\le1}|\lim_{i}\tau((u-\mathrm{i}(1-\E)(v))e_{i}a^{\ast}e_{i})|\\
        & \le\sup_{a\in \cup_{i\in I}\mathcal{A}_{e_{i}},\|a\|_1\le1}\sup_{i}\tau(|(u-\mathrm{i}(1-\E)(v))e_{i}a^{\ast}e_{i}|)\\
         & \le\sup_{a\in \cup_{i\in I}\mathcal{A}_{e_{i}},\|a\|_1\le1}\|(u-\mathrm{i}(1-\E)(v))a^{\ast}\|_1,
   \end{array}
   \ee
i.e.,
 \beq\label{eq:b_1(x)}
 \|x+\overline{\widetilde{y}}\|_{BMO}\le \sup_{a\in \cup_{i\in I}\mathcal{A}_{e_{i}},\|a\|_1\le1}\|(u-\mathrm{i}(1-\E)(v))a^*\|_{1}.
 \eeq
 Let $1<p\le\8$ and $\frac{1}{p}+\frac{1}{q}=1$. If $a\in \cup_{i\in I}\mathcal{A}_{e_{i}}$ and $\|a\|_1\le1$, then there exists $i_0\in I$ such that $a\in \mathcal{A}_{e_{i_0}}$.
Using Theorem 3.4 in \cite{BX}  we obtain that for $\e>0$, there exist $h_1\in H^p(\A_{e_{i_0}})$ and
$h_2\in H^q(\A_{e_{i_0}})$ such that
\be
a=h_2 h_1\quad\mbox{and}\quad \|h_1\|_p= 1,\,\|h_2\|_q\le 1+\e.
\ee
Hence,
\be
\begin{array}{rl}
\|(u-\mathrm{i}(1-\E)(v))a^*\|_{1}&\le \|(u-\mathrm{i}(1-\E)(v))h_1^*\|_{p}\|h_2^*\|_q\\
&\le(1+\e)\|(u-\mathrm{i}(1-\E)(v))h_1^*\|_{p}.
\end{array}
\ee
 Therefore, by \eqref{eq:b_1(x)},
 \be
 \|x+\overline{\widetilde{y}}\|_{BMO}\le(1+\e)\sup_{h\in H^p(\A),\|h\|_p\le1}\|(u-\mathrm{i}(1-\E)(v))h^*\|_{p},
 \ee
 and so
  \beq\label{eq:b_p(x)}
 \|x+\overline{\widetilde{y}}\|_{BMO}\le\sup_{h\in H^p(\A),\|h\|_p\le1}\|(u-\mathrm{i}(1-\E)(v))h^*\|_{p}.
 \eeq
 Applying \eqref{eq:b_1(x)} and \eqref{eq:b_p(x)}, we obtain the desired result.

 (ii) The right inequality is trivial. We
choose  $0 < \theta<1$ such that $1=\frac{1-\theta}{p}+\frac{\theta}{2}$. It is clear that,  we can view $u-i(1-\E)(v)$ as a bounded operator from $H^q(\A)$ to $L^q(\M)$ for all $0<q<\8$.  Then by interpolation, we get
\be\begin{array}{l}
\|u-\mathrm{i}(1-\E)(v)\|_{(H^{p}(\A),H^2(\A))_\theta\rightarrow(L^{p}(\M),L^2(\M))_\theta}\\
\qquad\le c\|u-\mathrm{i}(1-\E)(v)\|_{H^{p}(\A)\rightarrow L^{p}(\M)}^{1-\theta}\|u-\mathrm{i}(1-\E)(v)\|_{H^2(\A)\rightarrow L^2(\M)}^{\theta}
\end{array}
\ee
On the other hand, it is clear that
\be
\|u-i(1-\E)(v)\|_{H^{2}(\A)\rightarrow L^{2}(\M)}\le 2(\|x\|_{\infty}+\|v\|_{\infty})\le 4\|x+\overline{\widetilde{y}}\|_{BMO}.
\ee
Using Theorem \ref{thm:interpolation-semi-hp} and the left inequality in the case $1\le p\le\8$, we obtain that
 \be
  \|x+\overline{\widetilde{y}}\|_{BMO}\le c \Big(\sup_{h\in H^p(\A),\|h\|_p\le1}\|(u-\mathrm{i}(1-\E)(v))h^*\|_{p}\Big)^{1-\theta}(4\|x+\overline{\widetilde{y}}\|_{BMO})^{\theta}.
  \ee
  From this follows the left inequality. Thus the proof is complete.
\end{proof}

\section*{Appendix: Proof of Theorem \ref{interpolation-finite-hp} }

For easy reference,  using the Pisier's method in \cite{P}, we provide  a  proof of Theorem \ref{interpolation-finite-hp} ( It is stated in \cite{PX}
without proof, bar a comment that the Pisier's method may be extended to the noncommutative case. See the remark following Lemma 8.5 there).
 We denote by $L_0(\mathcal{M})$ the linear space of all $ \tau $-measurable operators.

Let $x\in L_0(\M)$.  We define
distribution function of $x$ by
$$
\lambda_{t}(x)=\tau(e_{(t,\infty)}(|x|)),\quad t\geq 0,
$$
where $e_{(t,\infty)}(|x|)$ is the spectral projection of $|x|$ corresponding
to the interval $(t,\infty)$.
Define
 $$  \mu_t(x)=\inf\{s>0\;:\; \lambda_s(x)\le t\},\quad t>0.$$
The function $t\mapsto \mu_t(x)$ is called  the   generalized singular
numbers  of $x$.
For more details on generalized singular value functions of measurable operators, see
\cite{FK}.

For $0<p<\8$, $0<q\le\8$,
the noncommutative Lorentz space $L^{p, q}(\M)$ is defined as the space of all measurable operators
$x$  such that
   $$
   \big\|x\big\|_{p, q}=\Big(\int_0^\8 \big[t^{1/p}\,
 \mu_t(x)\big]^q\,\frac{dt}t\Big)^{1/q}<\8.
 $$
Equipped with $\|\ \|_{p,q}$, $L^{p, q}(\M)$ is a
quasi-Banach space. Moreover, if $p>1$ and $q\ge1$, then $L^{p,
q}(\M)$ can be renormed into a Banach space. More precisely,
 $$x\mapsto \Big(\int_0^\8 \big[t^{-1\,+\,1/p}\,
 \int_0^t\mu_s(x)ds\big]^q\,\frac{dt}t\Big)^{1/q}$$
gives an equivalent norm on $L^{p, q}(\M)$. It is clear that $L^{p, p}(\M)=L^{p}(\M)$ (for details see \cite{DDP, Xu} ).

 The space $L^{p, \8}(\M)$ is usually called a weak
$L_p$-space, $0<p<\8$. It's quasi-norm admits the following useful
description in terms of the distribution function:
 $$
 \big\|x\big\|_{p, \8}
 =\sup_{s>0}s\,\lambda_s(x)^{1/p}
 $$
(see Proposition 3.1 in \cite{BCLJ}).

For $0<p<\8,\; 0<q\leq\infty$,  we define the noncommutative Hardy-Lorentz spaces
$$  H^{p,q}(\mathcal{A})=\mbox{closure
 of}\;\mathcal{A}\cap L^{p,q}(\mathcal{M})\; \mbox{in
 }\; L^{p,q}(\mathcal{M}),
 $$
 $$
H^{p,q}_{0}(\mathcal{A})=\mbox{closure
 of}\;\mathcal{A}_0\cap L^{p,q}(\mathcal{M})\; \mbox{in
 }\; L^{p,q}(\mathcal{M}).
$$

In the rest of  this appendix $(\M,\tau)$ always denotes a finite von Neumann algebra equipped with a
finite faithful trace $\tau$, and $\A$   a subdiagonal algebra
of  $\M$.

Let $\mathcal{N}$  be the algebra of infinite diagonal matrices with entries from
 $\mathcal{M}$ and  $\mathcal{B}$ be the algebra of infinite diagonal  matrices with entries from
 $\mathcal{A}$, i.e.

\be
\N=\left\{x:\;x= \left(
\begin{matrix}
x_{1} & 0 & \ldots \\
0& x_{2} & \ldots \\
\vdots & \vdots & \ddots
\end {matrix}
\right ),\; x_i\in \M, i\in \mathbb{N}\right\}
\ee
and
\be\B=\left\{a:\;a= \left(
\begin{matrix}
a_{1} & 0 & \ldots \\
0& a_{2} & \ldots \\
\vdots & \vdots & \ddots
\end {matrix}
\right ),\; a_i\in \A, i\in \mathbb{N}\right\}.
\ee
 For $x\in \mathcal{N}$ with entries $x_{i}$, define $\Phi(x)$ to be the diagonal matrix with entries
 $\mathcal{E}(x_{i})$ and
 $$
\nu(x)=\sum_{i\ge1}\tau(x_{i}).
$$
Then $(\mathcal{N},\nu)$ is a semifinite von Neumann algebra and
$\mathcal{B}$ is a semifinite subdiagonal algebra of $(\mathcal{N},\nu)$ respect to $\Phi$.

  Let $0<p_k, q_k\le\8$ $(k=0, 1)$ with $p_0\neq p_1$, and let
$0<\theta<1$ and $0<q\le\8$. Then
\be
 \hskip 2cm \big(L^{p_0, q_0}(\N),\; L^{p_1, q_1}(\N)\big)_{\theta,q}=L^{p, q}(\N)\hskip 3cm \mbox{\rm(A.1)}
\ee
with equivalent quasi-norms, where $1/p=(1-\theta)/p_0 + \theta/p_1$ (for details see \cite{DDP1,PX, Xu2}).

Using the same method as in the proof of Lemma 4.1 of \cite{P}, we obtain that

{\bf Lemma A.1} {\it Let $1<p<\8$. If $\theta=\frac{1}{p}$, then
\be
 (L^{1,\8}(\N)/J(H_0^{1,\8}(\B)), \N/J(\B_0))_{1-\theta,\8}= H^{p,\8}(\B)
\ee
with equivalent norms.}

\begin{proof}
Let $1<r<q<\8$.  Set $\eta=1-\frac{1}{r}, \gamma=1-\frac{1}{q}$. It is clear that  $0<\eta<\gamma<1$. Using (A.1) and Theorem 6.3 in \cite{B1}  we know that
\be
(L^1(\N), \N)_{\eta,r}=L^r(\N),\;(L^1(\N), \N)_{\gamma,q}=L^q(\N)
\ee
and
\be
(H^1(\B), \B)_{\eta,r}=H^r(\B),\;(H^1(\B), \B)_{\gamma,q}=H^q(\B).
\ee
On the other hand,  (i) and (iii) of Lemma 6.5 in \cite{B1}, we have that $(H^1(\B),H^r(\B))$  is $K$-closed with respect to $(L^1(\N),L^r(\N))$ and  $(H^q(\B),\B)$  is $K$-closed with respect to $(L^q(\N),\N)$. Hence by Theorem 1.2 in \cite{K}, we obtain $(H^1(\B),\B)$  is $K$-closed with respect to $(L^1(\N),\N)$. Since one can extend $\Phi$ to a contractive projection from $L^p(\N)$
onto $L^p(\B\cap\B^*)$ for every $1 \le p \le\8$, we deduce that $(H^1_0(\B),\B_0)$  is $K$-closed with respect to $(L^1(\N),\N)$. By Holmstedt's formula (see \cite{BL}, p. 52-53), we get $(H_0^{1,\8}(\B),\B_0)$ is $K$-closed with respect to $(L^{1,\8}(\N),\N)$. Therefore,
$$
(J(H_0^{1,\8}(\B)),J(\B_0))_{1-\theta,\8}=J(H_0^{q,\8}(B)).
$$
By Theorem 4.2 in \cite{Ja}, we have that
$$\begin{array}{l}
(L^{1,\8}(\N)/J(H_0^{1,\8}(\B)), \N/J(\B_0))_{1-\theta,\8}\\
\qquad=(L^{1,\8}(\N), \N)_{1-\theta,\8}/(J(H_0^{1,\8}(\B)),J(\B_0))_{1-\theta,\8} \\
 \qquad=L^{p,\8}(\N)/J(H_0^{p,\8}(\B))
  \end{array}
$$
with equivalence of quasi-norms. By (A.1) and Theorem 6.3 in \cite{B1}, for $1<p_0<p<p_1<\8$ and $0<\vartheta<1$ with $\frac{1}{p}=\frac{1-\vartheta}{p_0}+\frac{\vartheta}{p_1}$, we have that
\be
(L^{p_0}(\N), L^{p_1}(\N))_{\vartheta,\8}=L^{p,\8}(\N),\; (H^{p_0}(\B), H^{p_1}(\B))_{\vartheta,\8}=H^{p,\8}(\B).
\ee
Applying Theorem 4.2 in \cite{B1} and interpolation, we  deduce that
$$L^{p,\8}(\N)/J(H_0^{p,\8}(\B))$$
  can be identified with $H^{p,\8}(\B)$. Hence, the desired result holds. This completes the proof.
\end{proof}

For $1\le p\le \8$, we define $K_p:L^p(\M)\rightarrow L^{p,\8}(\N)$ by
\be
K_p(x)=\left(
\begin{matrix}
x & 0 &0& \ldots \\
0& 2^{-\frac{1}{p}}x & 0&\ldots \\
0& 0 & 3^{-\frac{1}{p}}x&\ldots \\
\vdots & \vdots & \vdots &\ddots
\end {matrix}
\right ),\quad \forall x\in L^p(\M).
\ee
Then
$$
 \hskip 5cm \|x\|_p=\|K_p(x)\|_{p,\8}. \hskip 5cm \mbox{\rm(A.2)}
$$
Indeed,
it is clear that for any  $x\in\M$,
 $$
 \|K_p(x)\|_\8=\|x\|_\8.
 $$
 If $1\le p< \8$ and $x\in L^{p}(\M)$,  then by  Chebychev¡¯s  inequality,
 $$
 \tau(e_{(j^\frac{1}{p}t,\8)}(|x|))\le \frac{\|x\|_p^p}{jt^p}<\8, \quad  \forall t>0,\; \forall j\ge1.
 $$
 Since $e_{(j^\frac{1}{p}t,\8)}(|x|)=e_{(j^\frac{1}{p}t,(j+1)^\frac{1}{p}t]}(|x|)+e_{((j+1)^\frac{1}{p}t,\8)}(|x|)$, we get
 $$
 \tau(e_{(j^\frac{1}{p}t,(j+1)^\frac{1}{p}t]}(|x|))=\tau(e_{(j^\frac{1}{p}t,\8)}(|x|))-\tau(e_{((j+1)^\frac{1}{p}t,\8)}(|x|)).
 $$
Using left-Riemann sums, we further get:
$$
\begin{array}{rl}\|x\|^p_p&=\tau(|x|^{p})
=\int_{0}^{+\infty}s^{p}d\tau(e_{s}(|x|))\\
& = \sup_{t>0}\sum_{j\ge1}jt^{p}\tau(e_{(j^\frac{1}{p}t,(j+1)^\frac{1}{p}t]}(|x|))\\
& = \sup_{t>0}\sum_{j\ge1}jt^{p}[\tau(e_{(j^\frac{1}{p}t,\8)}(|x|))-\tau(e_{((j+1)^\frac{1}{p}t,\8)}(|x|))]\\
& = \sup_{t>0}\sum_{j\ge1}t^{p}\tau(e_{(j^\frac{1}{p}t,\8)}(|x|))\\
&= \sup_{t>0}\sum_{j\ge1}t^{p}\tau(e_{(t,\8)}(|j^{-\frac{1}{p}}x|))\\
&=\sup_{t>0}t^p\nu(e_{(t,\8)}(K_p(x)))\\
&=\sup_{t>0}t^p\lambda_t(K_p(x))\\
& =\|K_p(x)\|_{p,\8}^p.
\end{array}
$$
Thus $K_p$ is an isometry. Note that $K_p$ maps $H_0^p(\A)$ into $H^{p,\8}_0(\B)$. Let
$$\widetilde{K}_p:L^p(\M)/J(H_0^p(\A))\rightarrow L^{p,\8}(\N)/J(H^{p,\8}_0(\B))$$
 be the mapping canonically associated to $K_p$. We have following result (see Lemma 4.2 in \cite{P}).

 {\bf Lemma A.2} {\it Let $1<p<\8$ and $\theta=\frac{1}{p}$. Then $\widetilde{K}_p$ defines a bounded mapping  from
$$
(L^{1}(\M)/J(H_0^{1}(\A)), \M/J(\A_0))_{1-\theta}
$$
into $H^{p,\8}(\B)$, and the norm of  $\widetilde{K}_p$  depends only on $p$.}

\begin{proof}
Let
$$
X=\left\{\left(
\begin{matrix}
x & 0 &0& \ldots \\
0& 2^{-\frac{1}{p}}x & 0&\ldots \\
0& 0 & 3^{-\frac{1}{p}}x&\ldots \\
\vdots & \vdots & \vdots &\ddots
\end {matrix}
\right ):\quad \forall x\in L^1(\M)\right\}.
$$
 We equip $X$ with the norm
 $$
 \left\|\left(
\begin{matrix}
x & 0 &0& \ldots \\
0& 2^{-\frac{1}{p}}x & 0&\ldots \\
0& 0 & 3^{-\frac{1}{p}}x&\ldots \\
\vdots & \vdots & \vdots &\ddots
\end {matrix}
\right )\right\|_{X}=\|x\|_1.
$$
 Then $X$ becomes a Banach space. It is clear that $X\subset L^{1,\8}(\N)$ and this inclusion has norm one. Let
 $$
 X_0=\left\{\left(
\begin{matrix}
x & 0 &0& \ldots \\
0& 2^{-\frac{1}{p}}x & 0&\ldots \\
0& 0 & 3^{-\frac{1}{p}}x&\ldots \\
\vdots & \vdots & \vdots &\ddots
\end {matrix}
\right ):\quad \forall x\in J(H^1_0(\M))\right\}.
$$
 Then we have that
\be
\hskip 4.5cm  X/X_0\subset L^{1,\8}(\N)/J(H_0^{1,\8}(\B)) \hskip 4.2cm \mbox{\rm(A.3)}
\ee
with norm one.

For any $z\in S$, let
\be
K_z(x)=\left(
\begin{matrix}
x & 0 &0& \ldots \\
0& 2^{z-1}x & 0&\ldots \\
0& 0 & 3^{z-1}x&\ldots \\
\vdots & \vdots & \vdots &\ddots
\end {matrix}
\right ),\quad \forall x\in \M.
\ee
Then $\{\widetilde{K}_z\}_{z\in \overline{S}}$ is an  analytic family of linear operators on
$$ (L^1(\M)/J(H_0^{1}(\A)))\cap (\M/J(\A_0))$$
into
$$
X/X_0+N/J(\B_0).
$$
From (A.2), it follows that if $\mathrm{Re}(z) = 0$, $\widetilde{K}_z$ is a contraction from
$$
L^1(\M)/J(H_0^{1}(\A))
$$
into $X/X_0$ and if $\mathrm{Re}(z) = 1$, it is a contraction from $\M/J(\A_0)$ into $N/J(\B_0)$. Hence, by Stein's interpolation theorem for analytic families of operators (see Theorem 1 in \cite{CJ}), we obtain that
 $\widetilde{K}_p$ is a contraction from
 $$
 (L^1(\M)/J(H_0^{1}(\A)), \M/J(\A_0))_{1-\frac{1}{p}}
 $$
 into
 $$(X/X_0,N/J(\B_0))_{1-\frac{1}{p}}.
 $$
 By Theorem 4.7.1 in \cite{BL}, $\widetilde{K}_p$ is a contraction  from
 $$
 (L^1(\M)/J(H_0^{1}(\A)), \M/J(\A_0))_{1-\frac{1}{p}}
 $$
  into
  $$
  (X/X_0,N/J(\B_0))_{1-\frac{1}{p},\8}.
  $$
  By (A.3), $\widetilde{K}_p$ is a contraction  from
 $$
 (L^1(\M)/J(H_0^{1}(\A)), \M/J(\A_0))_{1-\frac{1}{p}}
 $$
 into
 $$
 (L^{1,\8}(\N)/J(H_0^{1,\8}(\B)),N/J(\B_0))_{1-\frac{1}{p},\8}.
 $$
Using Lemma A.1, we obtain the desired result.  This completes the proof.

\end{proof}

{\bf Lemma A.3} {\it Let $1<q<\8$ . Then there is a bounded
natural inclusion
\be
\hskip 3cm (L^{1}(\M)/J(H_0^{1}(\A)), \M/J(\A_0))_{1-\frac{1}{q}}\subset H^{q}(\A), \hskip 3cm \mbox{\rm(A.4)}
\ee
and the norm of the inclusion  depends only on $q$.}

\begin{proof}
Let $x\in \A$. By (A.2) and Lemma A.2, we have that
$$\begin{array}{rl}

\|x\|_q&=\left\|\left(
\begin{matrix}
x & 0 &0& \ldots \\
0& 2^{-\frac{1}{q}}x & 0&\ldots \\
0& 0 & 3^{-\frac{1}{q}}x&\ldots \\
\vdots & \vdots & \vdots &\ddots
\end {matrix}
\right )\right\|_{q,\8}\\
&\leq C\|x\|_{(L^{1}(\M)/J(H_0^{1}(\A)), \M/J(\A_0))_{1-\frac{1}{q}}}.
\end{array}
$$
Hence, we get (A.4). This completes the proof.
\end{proof}

Let $(X_0, X_1)$ be a compatible couple of complex Banach  spaces. If  $X_0\cap X_1$ is dense in  $X_0$ and  $X_1$, then the dual of  $X_0\cap X_1$ can be identified with  $X_0^*+ X_1^*$ and the dual of  $X_0+X_1$ can identified with  $X_0^*\cap X_1^*$, which provides a scheme where $X_0^*$ and  $X_1^*$ are compatible.
We shall need the second Calder\'{o}n interpolation method: let us denote by $\G(X_0,
X_1)$  the family of all functions $g:\;\overline{S}\to X_0+X_1$
satisfying the following conditions:
 \begin{enumerate}[$\bullet$]
 \item $g$ is continuous on $\overline{S}$ and analytic in
 $S$;
  \item $\|g(z)\|_{X_0+X_1}\le C(1+|z|)$;
 \item $g(k+it_1)-g(k+it_2)\in X_k$ for $t_1,\;t_2\in\real,\;k=0,\,1$,  and
$$
\|g\|_{\G(X_0,X_1)}=\max_{k=0,1}\sup_{t_1,t_2\in\real,t_1\neq t_2}\left\|\frac{g(k+it_1)-g(k+it_2)}{t_1-t_2}\right\|_{X_k}<\8.
$$
 \end{enumerate}
The space $\G(X_0,X_1)$ reduced modulo the constant functions and equipped
with the norm above, is a Banach space and the second complex interpolation
spaces are defined by
$$
(X_0,\; X_1)^\theta=\left\{g'(\theta)\;:\;g\in\G(X_0,X_1)\right\},
$$
which are Banach spaces under the norm
 $$\|x\|_{(X_0,\; X_1)^\theta}=\inf\big\{\big\|g\big\|_{\G(X_0, X_1)}\;:\; g'(\theta)=x,\;
 g\in\G(X_0, X_1)\big\}.$$
Calder\'{o}n \cite{C} showed the inclusion $(X_0,\; X_1)_\theta\subseteq(X_0,\; X_1)^\theta$ and Bergh \cite{B}
 proved that
$$
\hskip 2.8cm \|x\|_{(X_0,\; X_1)^\theta}=\|x\|_{(X_0,\; X_1)_\theta},\quad \forall x\in(X_0,\; X_1)_\theta. \hskip 2.4cm \mbox{\rm(A.5)}
$$

 {\bf Proof of Theorem \ref{interpolation-finite-hp}.}
 We first consider the case $p_0=1,\;p_1=\8$.
 The inclusions
$(H^1(\A),\A)_\theta \subset(L^1(\M),\M)_\theta= L^p(\M)$
and
$ (H^1(\A),\A)_\theta \subset H^1(\A)$
imply $(H^1(\A),\A)_\theta\subset H^p(\A)$ (see Proposition 3.3 in \cite{BX}), where $\frac{1}{p}=1-\theta$.
Conversely, let $\frac{1}{p}+\frac{1}{q}=1$, i.e., $\theta=\frac{1}{q}$. By Lemma A.3, $(L^{1}(\M)/J(H_0^{1}(\A)), \M/J(\A_0))_{1-\theta}\subset H^{q}(\A)$. Using the duality theorem (Theorem 4.5.1 in \cite{BL}) we obtain that
$$
 H^p(\A)\subset(\A, (H^1(\A))^{**})^{1-\theta}=( (H^1(\A))^{**},\A)^{\theta}.
$$
Hence,
$$
\hskip 3cm \|a\|_{( (H^1(\A))^{**},\A)^{\theta}}\le C\|a\|_p,\quad \forall a\in  H^p(\A). \hskip 3.2cm \mbox{\rm(A.6)}
$$
Since $\A\subset H^1(\A)$, from the definition, we get that the norm in $(H^1(\A),\A)_{\theta}$ is the restriction of the norm in $( (H^1(\A))^{**},\A)_{\theta}$. By (A.5), we obtain that
$$
\hskip 2cm \|a\|_{( (H^1(\A))^{**},\A)^{\theta}}=\|a\|_{(H^1(\A),\A)_{\theta}},\quad \forall a\in (H^1(\A),\A)_{\theta}. \hskip 2cm \mbox{\rm(A.7)}
$$
Note that $A\subset  H^p(\A)$ and $\A\subset (H^1(\A),\A)_{\theta}$. Hence, by (A.6) and (A.7), we get
$$
\|a\|_{(H^1(\A),\A)_{\theta}}\le C\|a\|_p,\quad \forall a\in  A.
$$
Since  $A$ is dense $H^p(\A)$, using the above inequality and the method in the proof of Theorem \ref{thm:interpolation-semi-hp} we obtain that
$$
\|a\|_{(H^1(\A),\A)_{\theta}}\le C\|a\|_p,\quad \forall a\in   H^p(\A).
$$
Thus gives the desired result for the case $p_0=1$ and $p_1=\8$.

By what is
already proved and and the reiteration theorem (\cite{BL}, Theorem
4.6.1) we obtain that
for any $1\le p_0,\ p_1\le\8$ and $0<\theta<1$
$$
 \hskip 4cm (H^{p_0}(\A), H^{p_1}(\A))_\theta=H^{p}(\A),  \hskip 4.8cm \mbox{\rm(A.8)}
$$
with  equivalent norms, where $\frac{1}{p}=\frac{1-\theta}{p_0}+\frac{\theta}{p_1}$.

From this special case, we get the general case.
Indeed, since $(H^{p_0}(\A), H^{p_1}(\A))_\theta \subset(L^{p_0}(\M),L^{p_1}(\M))_\theta= L^p(\M)$
and $ (H^{p_0}(\A),H^{p_1}(\A))_\theta \subset H^{\min\{p_0,p_1\}}(\A)$, by Proposition 3.3 in \cite{BX}, we get
$(H^{p_0}(\A),H^{p_1}(\A))_\theta\subset H^p(\A)$. Conversely, we choose an integer $n$ such that
$\min\{np_0,\;np_0\}\ge1$. Let $x\in H^p(\A)$. By Theorem 3.4 in \cite{BX}, for $\e>0$, there exist $x_j\in H^{np}(\A)\;(j=1,2,\cdots,n)$ such that
 $$
 x=x_1x_2\cdots x_n, \quad  \|x_1\|_{np}\|x_2\|_{np}\cdots\|x_n\|_{np}<\|x\|_p+\e.
 $$
Since $\frac{1}{np}=\frac{1-\theta}{np_0}+\frac{\theta}{np_1}$, by (A.8),
$x_j\in(H^{np_0}(\A), H^{np_1}(\A))_\theta,\,(j=1,2,\cdots,n)$. Then there exist $f_j\in \F_0(H^{np_0}(\A), H^{np_1}(\A))$ such that
$f_j(\theta)=x_j$ and
 $$
 \big\|f_j\big\|_{\F_0(H^{np_0}(\A), H^{np_1}(\A))}<\|x_j\|_{(H^{np_0}(\A), H^{np_1}(\A))_\theta}+\e\le C\|x_j\|_{np}+\e,\quad j=1,\,2,\,\cdots,\,n.
 $$
Set $f=f_1f_2\cdots f_n$. Then $f\in \F_0(H^{p_0}(\A), H^{p_1}(\A))$ and $f(\theta)=x$. Hence
$$
\begin{array}{l}
  \|x\|_{(H^{p_0}(\A), H^{p_1}(\A))_\theta}\\
  \qquad\le\big\|f\big\|_{\F_0(H^{p_0}(\A), H^{p_1}(\A))}\\
  \qquad=\max\big\{
 \sup_{t\in\real}\big\|f(it)\big\|_{H^{p_0}(\A)}\,,\;
 \sup_{t\in\real}\big\|f(1+it)\big\|_{H^{p_1}(\A)}\big\} \\
 \qquad\le \max\big\{ \sup_{t\in\real}\Pi_{j=1}^n\big\|f_j(it)\big\|_{H^{np_0}(\A)}\,,\;
 \sup_{t\in\real}\Pi_{j=1}^n\big\|f_j(1+it)\big\|_{H^{np_1}(\A)}\big\}\\
 \qquad\le \Pi_{j=1}^n\max\big\{ \sup_{t\in\real}\big\|f_j(it)\big\|_{H^{np_0}(\A)}\,,\;
 \sup_{t\in\real}\big\|f_j(1+it)\big\|_{H^{np_1}(\A)}\big\}\\
 \qquad=\Pi_{j=1}^n\big\|f_j\big\|_{\F_0(H^{np_0}(\A), H^{np_1}(\A))}\\
 \qquad< \Pi_{j=1}^n(C\|x_j\|_{np}+\e)\\
 \qquad <C^n\|x\|_p+O(\e).
\end{array}
$$
Letting $\e\rightarrow0$ yields  $\|x\|_{(H^{p_0}(\A), H^{p_1}(\A))_\theta}\le C^n\|x\|_p$. Hence $(H^{p_0}(\A), H^{p_1}(\A))_\theta=H^{p}(\A)$ with  equivalent norms.

\subsection*{Acknowledgement} The authors are grateful to the anonymous referee for making helpful comments and suggestions, which have been incorporated into this version of the paper.

\end{document}